\documentclass[english]{smfart}
\usepackage{geometry} 
\newtheorem{theorem}{Theorem}[section]

\newtheorem{lemma}[theorem]{Lemma} 
\newtheorem{proposition}[theorem]{Proposition}
\newtheorem{remark}[theorem]{Remark}
\makeatletter
\numberwithin{equation}{section}
\makeatother

\usepackage{amsmath,hhline,epsfig,latexsym}
\usepackage{amssymb}
\usepackage{hyperref,color}

\newcommand{\eps}{\varepsilon}

\def\O{\mathcal O}
\def\A{\mathcal A}

\def\Hvec{\mathbf{H}}

\def\Lvec{\mathbf{L}}

\def\nvec{\mathbf{n}}

\def\pvec{\mathbf{p}}

\def\Vvec{\mathbf{V}}
\def\vvec{\mathbf{v}}

\def\xvec{\mathbf{x}}
\def\yvec{\mathbf{y}}
\def\wvec{\mathbf{w}}
\def\Yvec{\mathbf{Y}}

\def\Fvec{\mathbf{F}}

\def\zvec{\mathbf{z}}

\def\fvec{\mathbf{f}}

\title[Least-squares approach for controllability problem for PDEs]{About Least-squares type approach to address direct and controllability problems}
\date{06-10-2015}
\author{Arnaud M\"unch}
\thanks{
Laboratoire de Math\'{e}matiques, Universit\'{e} Blaise Pascal (Clermont-Ferrand~2), UMR CNRS 6620, Campus des C\'ezeaux,  63177~Aubi\`ere, France. e-mail: {\tt arnaud.munch@math.univ-bpclermont.fr}. }

\author{Pablo Pedregal}
\thanks{
E.T.S. Ingenieros Industriales. Universidad de Castilla La Mancha.
Campus de Ciudad Real (Spain). e-mail: {\tt pablo.pedregal@uclm.es}.  Research supported in part by
MTM2007-62945 of the MCyT (Spain), and PCI08-0084-0424 of the JCCM (Castilla-La Mancha).
}

\begin{document}

\maketitle

\begin{abstract}
We discuss the approximation of distributed null controls for partial differential equations. The main purpose is to determine an approximation of controls that drives
the solution from a prescribed initial state at the initial time to the zero target at a prescribed final time. As a non trivial example, we mainly focus on the Stokes system for which the existence of square-integrable controls have been obtained in \textit{[Fursikov \& Imanuvilov, Controllability of Evolution Equations, 1996])} via Carleman type estimates. We introduce a least-squares formulation of the controllability problem, and we show that it allows the construction of strong convergent sequences of functions toward null controls for the Stokes system. 
The approach consists first in introducing a class of functions satisfying \textit{a priori} the boundary conditions in space and time - in particular the null controllability condition at time $T$ -, and then finding among this class one element satisfying the system. This second step is done by minimizing a quadratic functional, among the admissible corrector functions of the Stokes system. We also discuss briefly the direct problem for the steady Navier-Stokes system. The method does not make use of any duality arguments and therefore avoid the ill-posedness of dual methods, when parabolic type equation are considered.

 \end{abstract}

\tableofcontents


\section{Introduction}

\label{intro}

 Let $\Omega\subset \mathbb{R}^N$, $N=2$ or $N=3$ be a bounded connected open set whose boundary $\partial\Omega$ is Lipschitz.
   Let $\omega\subset \Omega$ be a (small) nonempty open subset, and assume that $T>0$.
   We use the notation $Q_T=\Omega\times(0,T), q_T=\omega\times (0,T)$, $\Sigma_T=\partial\Omega \times(0,T)$ and we denote by $\nvec=\nvec(\xvec)$ the outward unit normal to $\Omega$ at any point $\mathbf{x}\in\partial\Omega$. Bold letters and symbols denote vector-valued functions and spaces; for instance $\Lvec^2(\Omega)$ is the Hilbert space of the functions $\vvec = (v_1,\dots,v_N)$ with $v_i \in L^2(\Omega)$ for all $i$. 
   
   This work is related to the null controllability problem for the non-stationary Stokes system
\begin{equation}\label{stokes}
\left\{
\begin{aligned}
 &    \yvec_t  - \nu\Delta \yvec + \nabla \pi = \fvec\;1_\omega, \quad  \nabla\cdot\yvec=0  \quad \text{in}\ \ Q_T \\
  &   \yvec = \textbf{0} \quad \text{on}\ \ \Sigma_T, \quad \yvec(\cdot,0) = \yvec_0 \quad \text{in}\ \ \Omega
\end{aligned}
\right.
\end{equation}
   which describes a viscous incompressible fluid flow in the bounded domain $\Omega$. We use as a control function the density of external forces 
$\fvec = \fvec(\xvec, t)$ concentrated in the arbitrary subdomain $\omega$ during the time interval $(0,T)$; $\yvec$ is the vector field of the fluid velocity, and $\pi$ is the scalar pressure. The real $\nu$ denotes the constant viscosity of the fluid. The symbol $1_\omega$ stands for the characteristic function of $\omega$. We introduce the following spaces
\begin{equation}
\label{spaces}
\begin{aligned}
& \Hvec=\{\boldsymbol{\varphi}\in \Lvec^2(\Omega): \nabla\cdot \boldsymbol{\varphi}=0\,\textrm{ in }\,\Omega, \,\,\boldsymbol{\varphi}\cdot\nvec=0\;\textrm{on}\;\partial\Omega\},\\
& \Vvec=\{\boldsymbol{\varphi}\in\Hvec_0^1(\Omega): \nabla\cdot\boldsymbol{\varphi}=0\;\textrm{in}\;\Omega\}, \quad U=\biggl\{\psi\in L^2(\Omega): \int_{\Omega}\psi(\xvec)\,d\xvec=0\biggr\}.
\end{aligned}
\end{equation}   
Then, for any $\yvec_0\in \Hvec,~ T > 0$, and $\fvec \in \Lvec^2(q_T)$, there exists exactly one
solution~$(\yvec,\pi)$ of (\ref{stokes}) with the following regularity : 
$$\yvec \in C^0\left([0, T];\Hvec\right)\cap L^2\left(0, T;\Vvec\right), \,\,\pi \in L^2(0,T; U)$$
 (see \cite{Temam}). The null controllability problem for~(\ref{stokes}) at time $T$ is the following: 
 \begin{center}
 {\it
   For any~$\yvec_0\in \Hvec$,  find~$\fvec\in \Lvec^2(q_T)$ such that the corresponding solution to~\eqref{stokes} satisfies
\begin{equation}\label{null_condition}
\yvec(\cdot,T) = \mathbf{0} \ \ \hbox{in}\ \ \Omega.
\end{equation}
}
\end{center}
The controllability properties of evolution PDEs have attracted a lot of works in the last decades: some relevant references are \cite{coron, F-I1, LT1, JLL2, Russell}. In particular, the Stokes - and more generally the Navier-Stokes - system has received a lot of attention: we mention the references \cite{FC-G-P,Ima-JPP-Yamamoto}. The following result is proved in 
\cite{F-I-93} (see also \cite{fabre, F-I1, Imanuvilov-NS}) by the way of Carleman estimates. 
\begin{theorem}[Fursikov-Imanuvilov]\label{th-1}
   The linear system \emph{(\ref{stokes})} is null-controllable at any time $T>0$.
\end{theorem}
On the other hand, the (numerical) approximation of controls either distributed or located on the boundary for the Stokes system has received much less attention, due to the underlying ill-posedness of the approximation. In practice, such approximation is usually addressed in the framework of an optimal control reformulation. Precisely, one seeks to minimize the quadratic functional  $J(\fvec):=\frac{1}{2}\Vert \fvec\Vert^2_{\Lvec^2(q_T)}$ over the non-empty set 
$$
\mathcal{C}(\yvec_0,T)=\{ (\yvec,\fvec): \fvec\in \Lvec^2(q_T),\, \yvec \;\textrm{solves} \; (\ref{stokes})\; \textrm{and satisfies}\; (\ref{null_condition})\}.
$$
Following \cite{glowinski08}, duality arguments allow to replace this constrained minimization by the unconstrained minimization of its conjugate function $J^{\star}$ defined as 
\begin{equation}
J^{\star}(\boldsymbol{\varphi_T})=\frac{1}{2}\int\!\!\!\!\int_{q_T} \vert \boldsymbol{\varphi}\vert^2 \,d\xvec\,dt + \int_{\Omega} \yvec_0\cdot \boldsymbol{\varphi}(\cdot,0)\,d\xvec \nonumber
\end{equation}
over $\boldsymbol{\varphi_T}\in\boldsymbol{\mathcal{H}}$, where $(\boldsymbol{\varphi},\sigma)$ solves the adjoint backward Stokes system associated with (\ref{stokes}) : 
\begin{equation}\label{stokes_adjoint}
\left\{
\begin{aligned}
 &    -\boldsymbol{\varphi}_t  - \nu\Delta \boldsymbol{\varphi} + \nabla \sigma = \boldsymbol{0}, \quad  \nabla\cdot \boldsymbol{\varphi}=0  \quad \text{in}\ \ Q_T \\
  &   \boldsymbol{\varphi} = \textbf{0} \quad \text{on}\ \ \Sigma_T, \quad \boldsymbol{\varphi}(\cdot,T) = \boldsymbol{\varphi_T} \quad \text{in}\ \ \Omega
\end{aligned}
\right.
\end{equation}

$\boldsymbol{\mathcal{H}}$ is the Hilbert space defined as the completion of any smooth space functions included in $\Hvec$ for the norm 
$\Vert \boldsymbol{\varphi}\Vert_{\Lvec^2(q_T)}$. The control of minimal square-integrable norm is then given by $\fvec=\boldsymbol{\hat{\varphi}}\, 1_{\omega}$ where $\boldsymbol{\hat{\varphi}}$ is associated with the unique minimizer $\boldsymbol{\hat{\varphi}_T}$ in $\boldsymbol{\mathcal{H}}$ of $J^{\star}$ (see \cite{coron, glowinski08}). 
The difficulty, when one wants to approximate such control in any finite dimensional space, that is when one likes to minimize numerically $J^{\star}$, is that the space $\boldsymbol{\mathcal{H}}$ is huge, in particular, contains $\cup_{s\in \mathbb{N}}\Hvec^{-s}(\Omega)$, and even elements that may not be distributions. 
Numerical experiments do suggest that the minimizer $\boldsymbol{\hat{\varphi}_T}$ is very singular (we refer to \cite{EFC-AM-I, EFC-AM-II, AM-EZ} for a detailed analysis in the close case of the heat equation). This phenomenon is independent of the choice of $J$, but is related to the use of dual variables. Actually, the equality (\ref{null_condition}) can be viewed as an equality in a very small space (due to the strong regularization effect of the heat kernel). Accordingly, the associated multiplier $\boldsymbol{\varphi_T}$ belong to a large dual space, much larger than $\Lvec^2(\Omega)$, that is hard to represent (numerically) in any finite dimensional space.


An alternative way of looking at these problems and avoiding the introduction of dual variables has been proposed in \cite{pedregal}. It is based on the following simple strategy. Instead of working all the time with solutions of the underlying state equation, and looking for one that may comply with the final desired state, one considers a suitable class of functions complying with required initial, boundary, final conditions and appropriate regularity, and seeks one of those that is a solution of the state equation. This is in practice accomplished by setting up an error functional defined for all feasible functions, and measuring how far those are from being a solution of the underlying state equation. 

One advantage of this variational approach is that the way to get closer to a solution of the problem is by minimizing a functional that cannot get stuck on local minima because the only critical points of the error turn out to be global minimizers with zero error. Therefore a general strategy for approximation consists in using a descent algorithm for this error functional. This approach which has the flavor of a least-squares type method has been employed successfully in our null controllability context for the linear heat equation in \cite{AM-PP} and for an hyperbolic system in \cite{munch-eect}. 

We apply this approach to the Stokes system. In Section \ref{sec2}, we describe the ingredients of the variational approach for the system (\ref{stokes}) and reduce the search of one controlled trajectory for the Stokes system to the minimization of the quadratic functional $E$ defined by (\ref{foncE}) over the affine space $\A$ defined by (\ref{spaceA}).
In Section \ref{sec3}, by a general-purpose lemma (Lemma \ref{basic}), using the very specific structure of the functional $E$, we prove that we may construct minimizing sequences for the error functional $E$ that do converge strongly to an extremal point for $E$ (see Proposition \ref{convergence}).  Section \ref{sec4}, we adapt the argument for the direct problem of the steady Navier-Stokes system. 

Results of this work were partially announced in the note \cite{AM-PP-crasstokes}.

\section{A least-squares reformulation}\label{sec2}

Following \cite{AM-PP,pedregal}, we define the non-empty space 
\begin{equation}\label{spaceA}
\begin{aligned}
 \A= \biggl\{(\yvec, \pi, \fvec); \;\; &\yvec\in
 \Lvec^2(0,T,\Hvec_0^1(\Omega)), \yvec_t\in \Lvec^2(0,T; \Hvec^{-1}(\Omega)), \\ 
 & \yvec(\cdot,0)=\yvec_0,\; \yvec(\cdot,T)=\textbf{0},\; \pi\in L^2(0,T; U),\; \fvec\in \Lvec^2(q_T)\biggr\}.
\end{aligned}
\end{equation}
These hypotheses on $\yvec$ imply that it belongs to $C([0,T], \Lvec^2(\Omega))$ and give a meaning to the equalities $\yvec(\cdot,0)=\yvec_0$, $\yvec(\cdot,T)=\textbf{0}$ in $\Lvec^2(\Omega)$. Note also that $\A$ is defined in agreement with the regularity of any solution $(\yvec,\pi)$ of the Stokes system with a source term $\fvec\in \Lvec^2(Q_T)$. Then, we define the functional $E:\A\to\mathbb{R}^+$ by 
\begin{equation}\label{foncE}
E(\yvec,\pi,\fvec)=\frac{1}{2}\int\!\!\!\!\int_{Q_T} (\vert \vvec_t\vert^2+ \vert\nabla \vvec\vert^2 + \vert\nabla\cdot\yvec\vert^2)\,d\xvec\,dt
\end{equation}
where the corrector $\vvec$ is the unique solution in $\Hvec^{1}(Q_T)$ of the (elliptic) boundary value problem
\begin{equation}\label{corrector_unsteady}
 \left\{
 \begin{aligned}
 & -\vvec_{tt}-\Delta \vvec + (\yvec_t-\nu \Delta \yvec + \nabla \pi -\fvec\,1_{\omega}) =0, \quad \textrm{in}\quad Q_T,\\
 & \vvec=0\quad \textrm{on}\quad\Sigma_T, \quad \vvec_t=0 \quad \textrm{on}\quad\Omega\times \{0,T\}.
 \end{aligned}
 \right.
 \end{equation}
 For any $(\yvec,\pi,\fvec)\in \A$, the term $\yvec_t-\nu \Delta \yvec + \nabla \pi -\fvec\,1_{\omega}$ belongs to $\Lvec^2(0,T; \Hvec^{-1}(\Omega))$ so that the functional $E$ is well-defined in $\A$. The approach developed here is based on the following result. 
 \begin{proposition}\label{Prop1}
 $(\yvec,\pi)$ is a controlled solution of the Stokes system (\ref{stokes}) by the control function $\fvec\,1_{\omega}\in \Lvec^2(q_T)$ if and only if $(\yvec,\pi,\fvec)$ is a solution of the extremal problem : 
\begin{equation} \label{extremal_pb}
\inf_{(\yvec,\pi,\fvec)\in \A} E(\yvec,\pi,\fvec).
\end{equation} 
\end{proposition}
\par\noindent\textsc{Proof-} From the controllability of the Stokes system  given by Theorem \ref{th-1}, the extremal problem (\ref{extremal_pb}) is well-posed in the sense that the infimum, equal to zero, is reached by any controlled solution of the Stokes system. Note that, without additional assumptions, the minimizer is not unique. Conversely, we check that any minimizer for $E$ is a solution of the (controlled) Stokes system:  let $(\Yvec,\Pi,\Fvec)\in \A_0$ be arbitrary where 
\begin{equation}\label{spaceA0}
\begin{aligned}
 \A_0= \biggl\{(\yvec, \pi, \fvec); \; &\yvec\in
 \Lvec^2(0,T,\Hvec_0^1(\Omega)),\; \yvec_t\in \Lvec^2(0,T; \Hvec^{-1}(\Omega)), \\ 
 & \yvec(\cdot,0)=\yvec(\cdot,T)=\textbf{0}, \;\pi\in L^2(0,T; U),\;\fvec\in \Lvec^2(q_T)\biggr\}.
\end{aligned}
\end{equation}
The first variation of $E$ at the point $(\yvec,\pi,\fvec)$ in the admissible direction $(\Yvec,\Pi,\Fvec)$
defined by
\begin{equation}
\begin{aligned}
\langle E^{\prime}(\yvec,\pi,\fvec),(\Yvec,\Pi,\Fvec)\rangle= \lim_{\eta\to 0} \frac{E( (\yvec,\pi,\fvec)+\eta (\Yvec,\Pi,\Fvec))-E(\yvec,\pi,\fvec)}{\eta},
\end{aligned}
\end{equation}
exists, and is given by 
\begin{equation}
\langle E^{\prime}(\yvec,\pi,\fvec),(\Yvec,\Pi,\Fvec)\rangle=\int\!\!\!\!\int_{Q_T} \biggl( \vvec_t\cdot \Vvec_t + \nabla\vvec\cdot \nabla\Vvec+ (\nabla\cdot \yvec)(\nabla\cdot\Yvec)    \biggr)\,d\xvec\,dt
\end{equation}
where the corrector $\Vvec\in \Hvec^1(Q_T)$, associated with $(\Yvec,\Pi,\Fvec)$, is the unique solution of 
\begin{equation}\label{corrector_unsteady_V}
 \left\{
 \begin{aligned}
 & -\Vvec_{tt}-\Delta \Vvec + (\Yvec_t-\nu \Delta \Yvec + \nabla \Pi -\Fvec\,1_{\omega}) =0 \quad \textrm{in}\quad Q_T,\\
 & \Vvec=0 \quad \textrm{on}\quad\Sigma_T, \quad\Vvec_t=0\quad \textrm{on}\quad\Omega\times \{0,T\}.
 \end{aligned}
 \right.
 \end{equation}
 Multiplying the main equation of this system by $\vvec$ (recall that $\vvec$ is the corrector associated with the minimizer $(\yvec, \pi, \fvec)$), integrating by parts, and using the boundary conditions on $\vvec$ and $\Vvec$, we get 
 %
%
\begin{equation}\label{Eprime}
 \begin{aligned}
 \langle E^{\prime}(\yvec,\pi,\fvec),(\Yvec,\Pi,\Fvec)\rangle =  -\int\!\!\!\!\int_{Q_T}& (-\Yvec\cdot\vvec_t + \nu \nabla\Yvec\cdot\nabla \vvec- \Pi\,\nabla\cdot\vvec-\Fvec\cdot\vvec\,1_{\omega})   \,d\xvec\,dt \\
 &+ \int\!\!\!\!\int_{Q_T}(\nabla\cdot\yvec)(\nabla \cdot \Yvec)\,d\xvec\,dt, \qquad \forall (\Yvec,\Pi,\Fvec)\in \A_0,
 \end{aligned}
 \end{equation}
 where we have used that  
 $$
 -\int_0^T \langle\Yvec_t,\vvec\rangle_{\Hvec^{-1}(\Omega),\Hvec^1(\Omega)} \,dt=\int\!\!\!\!\int_{Q_T} \Yvec\cdot \vvec_t\,d\xvec\,dt - \int_\Omega [\Yvec\cdot\vvec]_0^T\,d\xvec = \int\!\!\!\!\int_{Q_T} \Yvec\cdot\vvec_t\,d\xvec\,dt,
 $$ 
 and that   
 $$
 \int_0^T  \langle\nabla \Pi, \vvec\rangle_{\Hvec^{-1}(\Omega),\Hvec^1(\Omega)} dt = -\int\!\!\!\!\int_{Q_T} \Pi\;\nabla\cdot \vvec\,d\xvec\,dt. 
 $$
 Therefore if $(\yvec,\pi,\fvec)$ minimizes $E$, the equality $\langle E^{\prime}(\yvec,\pi,\fvec),(\Yvec,\Pi,\Fvec)\rangle=0$ for all $(\Yvec,\Pi,\Fvec)\in \A_0$ implies that the corrector $\vvec=\vvec(\yvec,\pi,\fvec)$ solution of (\ref{corrector_unsteady}) satisfies the conditions
\begin{equation}\label{corrector_unsteady_control}
 \left\{
 \begin{aligned}
 & \vvec_{t}+\nu\Delta \vvec -\nabla(\nabla\cdot \yvec)=0, \quad \nabla\cdot\vvec=0, & \textrm{in}\quad Q_T,\\
 & \vvec=0, & \textrm{in}\quad q_T.
 \end{aligned}
 \right.
 \end{equation}
 But from the unique continuation property for the Stokes system (see \cite{fabre, fabre-lebeau}), it turns out that $\vvec=0$ in $Q_T$ and that $\nabla\cdot\yvec$ is a constant in $Q_T$. Eventually, the relation (\ref{Eprime}) is then reduced to 
 $$
 \langle E^{\prime}(\yvec,\pi,\fvec),(\Yvec,\Pi,\Fvec)\rangle=   (\nabla\cdot\yvec)\int\!\!\!\!\int_{Q_T}\nabla \cdot \Yvec\,d\xvec\,dt =0, \quad \forall (\Yvec,\Pi,\Fvec)\in\A_0
 $$ 
and  then implies that this constant is zero. Consequently, if $(\yvec,\pi,\fvec)\in \A$ is a minimizer for $E$, then $\nabla\cdot\yvec=0$ in $Q_T$, and the corrector $\vvec$ is zero in $Q_T$, so that $E(\yvec,\pi,\fvec)=0$. Therefore, $(\yvec,\pi,\fvec)$ solves (\ref{stokes}), and since $(\yvec,\pi,\fvec)\in \A$, the state $\yvec$ is controlled at the time $T$ by the function $\fvec$ which acts as a control distributed in $\omega$. \qed

\vskip 0.1cm
\begin{remark}
The proof of Proposition \ref{Prop1} only utilizes optimality of $(\yvec,\pi,\fvec)$ and not its minimality. Therefore in the statement of the proposition, we could have written instead :
$(\yvec,\pi)$ is a controlled solution of the Stokes system (\ref{stokes}) by the control function $\fvec\,1_{\omega}\in \Lvec^2(q_T)$ if and only if $(\yvec,\pi,\fvec)$ is a stationary point for the functional $E(\yvec,\pi,\fvec)$ over  $(\yvec,\pi,\fvec)\in \A$. This is relevant from the perspective of the numerical simulation for it guarantees that the numerical procedure based on a descent strategy cannot get stuck in local minima.
\end{remark}
 \vskip 0.1cm
\begin{remark}
 For any $(\yvec,\pi,\fvec)\in\A$,  the cost $E$ can be formulate as follows
 $$
 E(\yvec,\pi,\fvec)=\frac{1}{2}\Vert \yvec_t-\nu\Delta\yvec+\nabla\pi-\fvec\,1_{\omega}\Vert^2_{\Hvec^{-1}(Q_T)} + \frac{1}{2}\Vert \nabla\cdot\yvec\Vert^2_{\Lvec^2(Q_T)}.
 $$
 This justifies the least-squares terminology. The use of least-squares type approaches to solve linear and nonlinear problem is not new and we refer to \cite{Bochev, glowinski83, gunzburger} for many applications. The use of least-squares type approaches in the controllability context comes from \cite{pedregal}.  
\end{remark}
 \vskip 0.1cm
\begin{remark} \label{divepsilon}
The quasi-incompressibility case is obtained in the same way. It suffices to add $\epsilon \pi$ (for any $\epsilon>0$) to the divergence term in the functional $E$. This is also in practice a classical trick to fix the constant of the pressure $\pi$ (see Section \ref{sec4}). 
\end{remark}
\vskip 0.1cm
\begin{remark}\label{compact_time_support}
The approach allows to consider compact support control jointly in time and space. It suffices to replace the function $1_{\omega}$ in (\ref{stokes}) by any compact support function in time and space, say $1_{\tilde{q}_T}$, where $\tilde{q}_T$ denotes a non-empty subset of $Q_T$. Since Theorem \ref{th-1} holds for any controllability time $T$ and any subset $\omega$ of $\Omega$, the controllability of (\ref{stokes}) remains true as soon as $\tilde{q}_T$ contains any non-empty cylindrical domain of the form $\omega_1\times (t_1,t_2) \subset \Omega\times (0,T)$. 
\end{remark}
\vskip 0.1cm
\begin{remark} \label{directproblem}
\textit{A fortiori}, the approach is well-adapted to address the direct problem (which consists, $\fvec$ being fixed, in solving the boundary value problem (\ref{stokes}): it suffices to remove from $\A$ and $\A_0$ the condition (\ref{null_condition}), and fix the forcing term $\fvec$ (see \cite{gunzburger} using a similar least-squares type point of view). In that case, the second line of (\ref{corrector_unsteady_control}) is replaced by $\vvec(\cdot,T)=0$, which implies with the first line, that $\vvec$ and $\nabla\cdot \yvec=0$ both vanish in $Q_T$. 
\end{remark}

It is worth to notice that this approach allows to treat at the same time the null controllability constraint and the incompressibility one. In this sense, the pair $(\pi, \fvec)$ can be regarded as a control function for the set of constraints
\begin{equation}
\yvec(\cdot,T)=0 \quad \textrm{on}\quad \Omega, \qquad \nabla\cdot \yvec=0 \quad \textrm{in}\quad Q_T.  \label{eq:contraint_gen}
\end{equation}
These two conditions are compatibles: there is no competition between them. In the uncontrolled situation, from the uniqueness, the pressure $\pi$ is unique as soon as the source term (here $\fvec 1_{\omega}$) is fixed. On the other hand, in our controllability context, the pair $(\pi,\fvec)$ is not unique: the pressure $\pi$ depends on $\fvec\,1_{\omega}$ and \textit{vice versa}. Therefore, the optimization with respect to both variables at the same time makes sense.  From this point of view, we may reformulate the problem as a general controllability problem for the heat equation: 
\[
\yvec_t-\nu \Delta \yvec = \boldsymbol{V} := \fvec\, 1_{\omega} - \nabla \pi \quad \textrm{in}\quad  Q_T, 
\]
$\boldsymbol{V}$ being a control function such that (\ref{eq:contraint_gen}) holds. The control function $\boldsymbol{V}$ acts on the whole domain, but on the other hand, should take the specific form $\boldsymbol{V}:= \fvec\,1_{\omega}-\nabla \pi$. 

Again, this perspective is different with the classical one, which consists in finding a control $\vvec\in \Lvec^2(q_T)$, such that $\yvec(\cdot,T)=0$ in $\Omega$ where $(\yvec,\pi)$  solves (\ref{stokes}).  This can done by duality, penalization technique, etc. Conversely, one may also consider iteratively first the null controllability constraint, that is, for any $\pi$ fixed in $L^2(0,T; U)$, find a control $\fvec_{\pi} \, 1_{\omega}$ such that (\ref{null_condition}) holds, and then find the pressure $\pi$ such that $\nabla\cdot \yvec_{\pi}=0$ holds in $Q_T$. Using again a least-squares type approach (for the heat equation, as developed in \cite{AM-PP}), the first step reduces to solve, for any $\pi\in L^2(0,T,U)$ fixed, the problem 
\[
\inf_{(\yvec_{\pi},\fvec_{\pi})\in \A_1} \tilde{E}(\yvec,\fvec):= \frac{1}{2}\Vert \vvec\Vert^2_{\Hvec^1(Q_T)}
\]
where $\vvec=\vvec(\yvec_{\pi},\pi,\fvec_{\pi})$ solves (\ref{corrector_unsteady}) and $\A_1$ is given by 
\[
 \A_1= \biggl\{(\yvec, \fvec); \;\; \yvec\in
 \Lvec^2(0,T,\Hvec_0^1(\Omega)), \yvec_t\in \Lvec^2(0,T; \Hvec^{-1}(\Omega)),  \yvec(\cdot,0)=\yvec_0,\; \yvec(\cdot,T)=\textbf{0},  \; \fvec\in \Lvec^2(q_T)\biggr\}.
\]
The second step consists in updating the pressure according to a descent direction for the function $G:L^2(0,T,U)\to \mathbb{R}$ defined by $G(\pi):=1/2 \Vert \nabla\cdot \yvec_{\pi}\Vert^2_{L^2(Q_T)}$.
We get that the first variation of $G$ at $\pi$ in the direction $\overline{\pi}\in L^2(0,T; U)$ is given by $<G^{\prime}(\pi), \overline{\pi}> = \int\!\!\!\int_{Q_T} \nabla\overline{\pi}\cdot \pvec\,dx\,dt$ where $\pvec$ solves
\[
-\pvec_t-\nu \Delta \pvec =\nabla (\nabla\cdot \yvec_{\pi}) \quad\textrm{in}\quad Q_T, \quad \pvec(\cdot,T)=0\quad\textrm{in}\quad\Omega,\quad \pvec=0\quad \textrm{on} \quad\Sigma_T.
\]
Again, this direct problem may be solved within the variational approach developed in this work (see Remark \ref{directproblem}).

\section{Convergence of some minimizing sequences for $E$} \label{sec3}

Proposition \ref{Prop1} reduces the approximation of a null control for (\ref{stokes}) to a minimization of the functional $E$ over $\A$. 
As a preliminary step, since $\A$ is not a vectorial space, we remark that any element of $\A$ can be written as the sum of one element of $\A$, say $\mathbf{s}_{\A}$, plus any element of the vectorial space $\A_0$. Thus, we consider for any ${\bf s_\A}:=(\yvec_\A,\pi_\A,\fvec_\A)\in \A$ the following problem: 
\begin{equation}
\min_{(\yvec,\pi,\fvec)\in \A_0}   E_{\bf s_\A}(\yvec,\pi,\fvec), \qquad E_{\bf s_\A}(\yvec,\pi,\fvec):= E({\bf s_\A} + (\yvec,\pi,\fvec)).   \label{extremal_pb_bis}
\end{equation}
Problems (\ref{extremal_pb}) and (\ref{extremal_pb_bis}) are equivalent. Any solution of Problem (\ref{extremal_pb_bis}) is a solution of the initial problem (\ref{extremal_pb}).
Conversely, any solution of Problem (\ref{extremal_pb}) can be decomposed as the sum $\mathbf{s}_\A + \mathbf{s}_{\A_0}$, for some $\mathbf{s}_{\A_0}$ in $\A_0$.

We endow the vectorial space $\A_0$ with its natural norm $\Vert \cdot\Vert_{\A_0}$ such that : 
\begin{equation}
\Vert \yvec,\pi,\fvec \Vert^2_{\A_0}:=  \int\!\!\!\!\int_{Q_T} (\vert \yvec\vert^2 + \vert\nabla \yvec\vert^2) \, d\xvec\,dt + \int_0^T \Vert \yvec_t(\cdot,t)\Vert^2_{\Hvec^{-1}(\Omega)} dt + \int\!\!\!\!\int_{Q_T} ( \vert \fvec\vert^2 + \vert \pi\vert^2 )\, d\xvec\,dt,
\end{equation}
recalling that $\Vert \yvec_t\Vert_{\Hvec^{-1}(\Omega)}=\Vert \bf{g}\Vert_{H_0^1(\Omega)}$ where $\mathbf{g}\in \Hvec_0^1(\Omega)$ solves $-\mathbf{\Delta} \mathbf{g} = \yvec_t$ in $\Omega$. We denote $\langle , \rangle_{\A_0}$ the corresponding scalar product. $(\A_0, \Vert\cdot\Vert_{\A_0})$ is an Hilbert space.

 The relation (\ref{Eprime}) allows to define a minimizing sequence in $\A_0$ for $E_{\bf s_\A}$. 
 
It turns out that minimizing sequences for $E_{\mathbf{s}_\A}$
 which belong to a precise subset of $\A_0$ remain bounded uniformly. This very valuable property is not \textit{a priori} guaranteed from the definition of  $E_{\mathbf{s}_\A}$. The boundedness of $E_{\mathbf{s}_\A}$ implies only the boundedness of the corrector $\vvec$ for the $\Hvec^1(Q_T)$-norm and the boundedness of the divergence $\nabla\cdot\yvec$ of the velocity field for the $L^2(Q_T)$-norm. Actually, this property is due to the fact that the functional $E_{\mathbf{s}_\A}$ is invariant in the subset of $\A_0$ which satisfies the state equations of (\ref{stokes}). 
 
In order to construct a minimizing sequence bounded in $\A_0$ for $E_{\mathbf{s}_\A}$, we introduce the linear continuous operator $\mathbf{T}$ which maps a triplet $(\yvec,\pi,\fvec)\subset \A$ into the corresponding vector $(\vvec, \nabla\cdot \yvec)\in \Hvec^1(Q_T)\times L^2(Q_T)$, with the corrector $\vvec$ as defined by (\ref{corrector_unsteady}). Then we define the space $A=\textrm{Ker}\,\mathbf{T} \cap \A_0$ composed of the elements $(\yvec,\pi,\fvec)$ satisfying the Stokes system and such that $\yvec$ vanishes on the boundary $\partial Q_T$. Note that $A$ is not the trivial space : it suffices to consider the difference of two distinct null controlled solutions of (\ref{stokes}). Finally, we note $A^\perp = (\textrm{Ker}\, \mathbf{T} \cap \A_0)^{\perp}$ the orthogonal complement of $A$ in $\A_0$ and $P_{A^\perp}: \A_0\to A^{\perp}$ the (orthogonal) projection on $A^{\perp}$.

We then define the following minimizing sequence  ${(\yvec^k,\pi^k,\fvec^k)}_{k\geq 0}\in A^\perp$ as follows: 
\begin{equation}
\label{minimizing_sequence}
\left\{
\begin{aligned}
& (\yvec^0,\pi^0,\fvec^0)\,\,\textrm{given in}\,\,A^\perp, \\
& (\yvec^{k+1},\pi^{k+1},\fvec^{k+1}) = (\yvec^{k},\pi^{k},\fvec^{k}) - \eta_k P_{A^{\perp}}(\overline{\yvec}^k,\overline{\pi}^k,\overline{\fvec}^k), \quad k\geq 0 
\end{aligned}
\right.
\end{equation}
where $(\overline{\yvec}^k,\overline{\pi}^k,\overline{\fvec}^k)\in \A_0$ is defined as the unique solution of the formulation
 \begin{equation}\label{FVdes}
 \langle(\overline{\yvec}^k,\overline{\pi}^k,\overline{\fvec}^k),(\Yvec,\Pi,\Fvec)\rangle_{\A_0}=  \langle E_{\bf s_0}^{\prime}(\yvec^{k},\pi^{k},\fvec^{k}),(\Yvec,\Pi,\Fvec)\rangle, \quad \forall (\Yvec,\Pi,\Fvec)\in \A_0.
 \end{equation}
 $\eta_k$ denotes a positive descent step. In particular, (\ref{FVdes}) implies that $\overline{\pi}^k=-\nabla\cdot\vvec^k\in L^2(Q_T)$ and $\overline{\fvec}^{k}=-\vvec^{k}\,1_{\omega}\in \Lvec^2(q_T)$ (actually in $\Hvec^1(q_T)$). 
 
 One main issue of our variational approach is to establish the convergence of the minimizing sequence defined by (\ref{minimizing_sequence}). We have the following result.
\vskip 0.1cm
\begin{proposition}\label{convergence}
For any $\bf{s_\A}\in \A$ and any $\{\yvec^0,\pi^0,\fvec^0\}\in A^\perp$, the sequence $\bf{s_\A} + \{(\yvec^k,\pi^k,\fvec^k)\}_{k\geq 0} \in \A$ converges strongly to a solution of the extremal problem (\ref{extremal_pb}).
\end{proposition}
\vskip 0.1cm
This proposition is the consequence of the following abstract result which can be adapted to many different situations where this variational perspective can be of help. 
\vskip 0.1cm
 \begin{lemma}\label{basic}
Suppose $\mathbf{T}:X\mapsto Y$ is a linear, continuous operator between Hilbert spaces, and $H\subset X$, a closed subspace, $u_0\in X$. Put
$$
E: u_0+H\mapsto\mathbb{R}^+,\quad E(u)=\frac12\|\mathbf{T}u\|^2,\quad A=\hbox{Ker}\,\mathbf{T}\cap H.
$$
\begin{enumerate}
\item 
$E: u_0+A^\perp\to\mathbb{R}$ is quadratic, non-negative, and strictly convex, where $A^\perp$ is the orthogonal complement of $A$ in $H$. 
\item If we regard $E$ as a functional defined on $H$, $E(u_0+\cdot)$, and identify $H$ with its dual, then the derivative $E'(u_0+\cdot)$ always belongs to $A^\perp$. In particular, a typical steepest descent procedure for $E(u_0+\cdot)$ will always stay in the manifold $u_0+A^\perp$.
\item If, in addition, $\min_{u\in H} E(u_0+u)=0$, then the steepest descent scheme will always produce sequences converging (strongly in $X$) to a unique (in $u_0+A^\perp$) minimizer $u_0+\overline u$ with zero error.
\end{enumerate}
\end{lemma}
\par\noindent
\textsc{Proof of Lemma \ref{basic}-} Suppose there are $u_i\in A^\perp$, $i=1, 2$, such that 
$$
E\left(u_0+\frac12u_1+\frac12u_2\right)=\frac12E(u_0+u_1)+\frac12E(u_0+u_2).
$$
Due to the strict convexity of the norm in a Hilbert space, we deduce that this equality can only occur if 
$\mathbf{T}u_1=\mathbf{T}u_2$. So therefore $u_1-u_2\in A\cap A^\perp=\{0\}$, and $u_1=u_2$. For the second part, note that for arbitrary $U\in A$, $\mathbf{T}U=0$, and so 
$$
E(u_0+u+U)=\frac12\|\mathbf{T}u_0+\mathbf{T}u+\mathbf{T}U\|^2=\frac12\|\mathbf{T}u_0+\mathbf{T}u\|^2=E(u_0+u).
$$
Therefore the derivative $E'(u_0+u)$, the steepest descent direction for $E$ at $u_0+u$, has to be orthogonal to all such $U\in A$. 

Finally, assume $E(u_0+\overline u)=0$. It is clear that this minimizer is unique in $u_0+A^\perp$ (recall the strict convexity in (i)). This, in particular, implies that for arbitrary $u\in A^\perp$, 
\begin{equation}\label{desigualdad}
\langle E'(u_0+u), \overline u-u\rangle\le0,
\end{equation}
because this inner product is the derivative of the section $t\mapsto E(u_0+t\overline u+(1-t)u)$ at $t=0$, and this section must be a positive parabola with the minimum point at $t=1$. If we consider the gradient flow 
$$
u'(t)=-E'(u_0+u(t)),\quad t\in[0, +\infty),
$$
then, because of (\ref{desigualdad}), 
$$
\frac d{dt}\left(\frac12\|u(t)-\overline u\|^2\right)=\langle u(t)-\overline u, u'(t)\rangle=
\langle u(t)-\overline u, -E^{\prime}(u_0+u(t))\rangle\le0.
$$
This implies that sequences produced through a steepest descent method will be minimizing for $E$, uniformly bounded in $X$ (because $\|u(t)-\overline u\|$ is a non-increasing function of $t$), and due to the strict convexity of $E$ restricted to $u_0+A^\perp$, they will have to converge towards the unique minimizer $u_0+\overline u$. 
\qed
\vskip 0.1cm
\begin{remark}
Despite the strong convergence in this statement, it may not be true that the error is coercive, even restricted to $u_0+A^\perp$, so that strong convergence could be very slow. Because of this same reason, it may be impossible to establish rates of convergence for these minimizing sequences. 
\end{remark}
\vskip 0.1cm
The element $u_0$ determines the non-homogeneous data set of each problem: source term, boundary conditions, initial and/or final condition, etc. The subspace $H$ is the subset of the ambient Hilbert space $X$ for which the data set vanishes. $\mathbf{T}$ is the operator defining the corrector, so that $\textrm{Ker}\,\mathbf{T}$ is the subspace of all solutions of the underlying equation or system. The subspace $A$ is the subspace of all solutions of the problem with vanishing data set. In some situations $A$ will be trivial, but in some others will not be so. 
The important property is (iii) in the statement guaranteeing that we indeed have strong convergence in $X$ of iterates. The main requirement for this to hold is to know, \textit{a priori}, that the error attains its minimum value zero somewhere, which in the situation treated here is guaranteed by  Theorem \ref{th-1}. 
\vskip 0.1cm
\textsc{Proof of Proposition \ref{convergence}-} The result is obtained by applying Lemma \ref{basic} as follows. If we put $B=\{\yvec\in \Lvec^2(0,T,\Hvec_0^1(\Omega)): \yvec_t\in \Lvec^2(0,T; \Hvec^{-1}(\Omega))\},$
  $X$ is taken to be $B\times L^2(0, T; U)\times \Lvec^2(q_T)$. $H$ is taken to be $\A_0$ as given in (\ref{spaceA0}) and $u_0=\boldsymbol{s_\A}\in \mathcal{A}\subset X$. The operator $\mathbf{T}$ maps a triplet $(\yvec, \pi, \fvec)\in \A\subset X$ into $(\vvec,\nabla\cdot \yvec)\in Y:=\Hvec^1(Q_T)\times L^2(Q_T)$ as explained earlier.  

\begin{remark}
The construction of the minimizing sequence only requires the resolution of standard well-posed elliptic problems over $Q_T$, well-adapted to general situations (time dependent support, mesh adaptation, etc). On the other hand, it is important to highlight that the $\Lvec^2(q_T)$ control function $\fvec$ obtained from the minimizing procedure does not \textit{a priori} minimize any specific norm (for instance the $\Lvec^2$-norm).  
\end{remark}

Without the projection on $(\textrm{Ker}\,\mathbf{T}\cap \A_0)^{\perp}$ in (\ref{minimizing_sequence}), the sequence $(\yvec^k,\pi^k,\fvec^k)$ remains a minimizing sequence for $E_{\mathbf{s}_A}$: actually, the values of the cost $E_{\mathbf{s}_\A}$ along the sequence $(\yvec^k,\pi^k,\fvec^k)$ are equal with or without the projection. This is due to the fact that the component of the descent direction $(\overline{\yvec}^k,\overline{\pi}^k,\overline{\fvec}^k)$ on $(\textrm{Ker}\,\mathbf{T}\cap \A_0)$ does not affect the value of the cost : on the other hand, without the projection, the minimizing sequence may not be bounded uniformly in $\A_0$, in particular the control function $\fvec$ may not be bounded in $\Lvec^2(q_T)$. 

The subset $A^\perp$ is not explicit, so that in practice the projection $P_{A^{\perp}}(\overline{\yvec}^k,\overline{\pi}^k,\overline{\fvec}^k)$ may be defined by $P_{A^{\perp}}(\overline{\yvec}^k,\overline{\pi}^k,\overline{\fvec}^k)=(\overline{\yvec}^k,\overline{\pi}^k,\overline{\fvec}^k)-\boldsymbol{p}$, where $\boldsymbol{p}$ solves the extremal problem : 
\begin{equation}
\min_{\boldsymbol{p}\in A} \Vert \boldsymbol{p} - (\overline{\yvec}^k,\overline{\pi}^k,\overline{\fvec}^k) \Vert_{\A_0}.   \label{projection_pb}
\end{equation}
Recalling that $A$ is by definition the set of triplets $(\yvec,\pi,\fvec)$ satisfying $\yvec_t-\nu\Delta \yvec+\nabla \pi - \fvec \,1_{\omega}=0$, $\nabla\cdot \yvec=0$ in $Q_T$ such that $\yvec$ vanishes on $\partial Q_T$, this extremal problem is nothing else than a controllability problem for the Stokes system, similar to the one considered in this work. Therefore, we shall bypass this projection and shall introduce instead a stopping criteria for the descent method measuring how far from $A^{\perp}$ the descent direction is.  

We refer to \cite{munch-mcss} for numerical experiments. 

\section{Direct problem for the steady Navier-Stokes system}\label{sec4}

The least-squares approach allows to address non-linear problem. We consider here simply the steady Navier-Stokes and address the direct problem: find $(\yvec,\pi)$ solution of 
\begin{equation}\label{navierstokes}
\left\{
\begin{aligned}
 &     - \nu\Delta \yvec + (\yvec\cdot \nabla)\yvec + \nabla \pi = \fvec, \quad  \nabla\cdot\yvec=0  \quad \text{in}\ \ \Omega \\
  &   \yvec = \textbf{0} \quad \text{on}\ \ \partial\Omega.
\end{aligned}
\right.
\end{equation}
  We recall the following result  (see \cite{Temam}) :  
 \begin{theorem}\label{th-Temam}
  For any $\fvec\in \Hvec^{-1}(\Omega)$, there exists at least one $(\yvec,\pi)\in \Hvec_0^1(\Omega)\times L^2_0(\Omega)$ solution of (\ref{navierstokes}). Moreover, if $\nu^{-2}\Vert \fvec\Vert_{\Hvec^{-1}(\Omega)}$ is small enough, then the couple $(\yvec,\pi)$ is unique. 
\end{theorem}

In order to solve this boundary value problem, we use a least-squares type approach. We consider the space  $\A=\Hvec_0^1(\Omega)\times L^2(\Omega)$ and then we define the functional $E:\A\to\mathbb{R}^+$ by 
\begin{equation}\label{foncEbis}
E(\yvec,\pi):=\frac{1}{2}\int_{\Omega} (\vert\nabla \vvec\vert^2 + \vert\nabla\cdot\yvec\vert^2)\,d\xvec
\end{equation}
where the corrector $\vvec$ is the unique solution in $\Hvec_0^{1}(Q_T)$ of the (elliptic) boundary value problem
\begin{equation}\label{corrector_unsteadybis}
 \left\{
 \begin{aligned}
 & -\Delta \vvec + (-\nu \Delta \yvec + div(\yvec \otimes \yvec) + \nabla \pi -\fvec) =0, \quad \textrm{in}\quad \Omega,\\
 & \vvec=0\quad \textrm{on}\quad\partial\Omega.
 \end{aligned}
 \right.
 \end{equation}

We then consider the following extremal problem
 
\begin{equation} \label{extremal_pbbis}
\inf_{(\yvec,\pi)\in \A} E(\yvec,\pi).
\end{equation} 

We recall the following equality
\begin{lemma}\label{lemma1}
\begin{itemize}
\item  For all $\yvec,\zvec$, $div(\yvec\otimes \zvec)=\yvec  \nabla\cdot \zvec + (\nabla \yvec) \zvec.$
\item For all $\yvec, \zvec, \pvec \in \Hvec_0^1(\Omega)$, 
\begin{equation}
\begin{aligned}
&\int_{\Omega} \biggl(div(\yvec\otimes \zvec) +  div(\zvec\otimes \yvec)\biggr)\cdot \pvec = -\int_{\Omega} (\yvec\otimes \zvec + \zvec\otimes \yvec):\nabla \pvec
\\
&= - \int_{\Omega}  \biggl((\nabla \pvec +(\nabla \pvec)^T)\yvec\biggr)\cdot \zvec 
\end{aligned}
\end{equation}
\end{itemize}
\end{lemma}

From Theorem \ref{th-Temam}, the infimum is equal to zero and is reached by an element solution of (\ref{navierstokes}). Conversely, we would like to state that the only critical point for $E$ correspond to solution of (\ref{navierstokes}). 

In view of \cite[Corollary 1.8]{pedregalSEMA}, let us first prove the following result.

\begin{proposition}\label{functionalerror} $E$ is an error functional, that is $E:\A\to \mathbb{R}$ is a $C^1$-functional over the Hilbert space $\A$, $E$ is non-negative and 
$$
\lim E(\yvec,\pi)=0 \quad \textrm{as}\quad  E^{\prime}(\yvec,\pi)\to 0.
$$
\end{proposition}

\begin{proof}
We write that 
\begin{equation}
E^{\prime}(\yvec,\pi)\cdot (\Yvec,\Pi)=\int_\Omega \nabla \vvec\cdot \nabla \Vvec + (\nabla\cdot \yvec)(\nabla\cdot \Yvec) dx
\end{equation}
where $\Vvec\in \Hvec_0^1(\Omega)$ solves 

\begin{equation}\label{corrector_unsteadyqua}
 \left\{
 \begin{aligned}
 & -\Delta \Vvec + (-\nu \Delta \Yvec + div(\yvec \otimes \Yvec)+ div(\Yvec \otimes \yvec)+ \nabla \Pi) =0, \quad \textrm{in}\quad \Omega,\\
 & \Vvec=0\quad \textrm{on}\quad\partial\Omega.
 \end{aligned}
 \right.
 \end{equation}
 Multiplying the state equation by $\vvec$ and using (\ref{lemma1}), we get
 \begin{equation}
 \begin{aligned}
E^{\prime}(\yvec,\pi)\cdot (\Yvec,\Pi)=&\int_\Omega \nu \Delta \vvec \cdot\Yvec + (\nabla \vvec+(\nabla \vvec)^T) \yvec )\cdot \Yvec + (\nabla\cdot\vvec)\Pi  d\xvec\\
+ & \int_{\Omega} (\nabla\cdot \yvec)(\nabla\cdot \Yvec) d\xvec
\end{aligned}
\end{equation}

or equivalently
 \begin{equation}\label{eprime}
 \begin{aligned}
E^{\prime}(\yvec,\pi)\cdot (\Yvec,\Pi)=&\int_\Omega -\nu \nabla \vvec \cdot\nabla\Yvec + (\yvec\otimes \Yvec+\Yvec\otimes \yvec):\nabla \vvec+ (\nabla\cdot\vvec)\Pi  d\xvec\\
+ & \int_{\Omega} (\nabla\cdot \yvec)(\nabla\cdot \Yvec) d\xvec.
\end{aligned}
\end{equation}

We then check that we can take $\Yvec=\vvec$, i.e.  $\vvec=\vvec(\yvec,\pi)\in \Hvec_0^1(\Omega)$ uniquely given by the solution of (\ref{corrector_unsteadyqua}) remains bounded with respect to $(\yvec,\pi)\in \Hvec_0^1\times L^2(\Omega)$. By definition, the corrector $\vvec$ solves the variational formulation: 
\begin{equation}
\int_\Omega ((\nabla \vvec+\nu\nabla \yvec - \yvec\otimes \yvec):\nabla \wvec -\pi \nabla\cdot\wvec -\fvec\cdot\wvec) =0, \quad \forall \wvec\in \Hvec_0^1(\Omega)
\end{equation}
Taking $\wvec=\vvec$, we get that 
\begin{equation}
\int_{\Omega} \vert\nabla \vvec\vert^2 d\xvec=\int_{\Omega}  (\yvec\otimes \yvec:\nabla\vvec - \nu \nabla \yvec:\nabla \vvec +\pi \nabla\cdot \vvec +\fvec\cdot \vvec)d\xvec
\end{equation}
so that, in view of the Poincar\'e inequality, 
\begin{equation}\label{estimate}
\begin{aligned}
\Vert \vvec\Vert_{\Hvec_0^1(\Omega)} \leq C (\Vert \yvec \otimes \yvec\Vert_{L^2(\Omega)}+ \Vert \yvec\Vert_{\Hvec_0^1(\Omega)} + \Vert \pi\Vert_{L^2(\Omega)}+\Vert \fvec\Vert_{\Hvec^{-1}(\Omega)}).
\end{aligned}
\end{equation}
leading to the result. Then, taking $\Yvec=\vvec$ in (\ref{eprime}), we get 
\begin{equation}
 \begin{aligned}
E^{\prime}(\yvec,\pi)\cdot (\vvec,\Pi)=&\int_\Omega -\nu \vert \nabla \vvec \vert^2 - (\yvec\otimes \vvec+\vvec\otimes \yvec):\nabla \vvec+ (\nabla\cdot\vvec)\Pi  d\xvec\\
+ & \int_{\Omega} (\nabla\cdot \yvec)(\nabla\cdot \vvec) d\xvec \\
=& \int_\Omega -\nu \vert \nabla \vvec \vert^2 -   (\vvec\otimes \vvec):\nabla \yvec + \frac{1}{2} (\nabla\cdot \yvec)\vert \vvec\vert^2  d\xvec\\
+ & \int_{\Omega} (\nabla\cdot \vvec)(\nabla\cdot \yvec +\yvec\cdot \vvec +\Pi) d\xvec \\
\end{aligned}
\end{equation}

Similarly, in view of (\ref{estimate}),  $\Pi_s=-(\nabla\cdot \yvec +\yvec\cdot \vvec)\in L^2(\Omega)$ remains bounded with respect to $(\yvec,\pi)$ and we write
\begin{equation}\label{Eprime_Pis}
E^{\prime}(\yvec,\pi)\cdot (\vvec,\Pi_s)= \int_\Omega -\nu \vert \nabla \vvec \vert^2 -   (\vvec\otimes \vvec):\nabla \yvec + \frac{1}{2} (\nabla\cdot \yvec)\vert \vvec\vert^2  d\xvec
\end{equation}

We then use the following result (consequence of the well-posedness of the Oseen equation)

\begin{lemma}\label{osseen}
For any $\yvec\in \Hvec_0^1(\Omega)$, $\Fvec\in L^2(\Omega)$, there exists $(\Yvec,\Pi)\in H_0^1(\Omega)\times L^2(\Omega)$  with $\nabla\cdot\Yvec=0$ such that 
\begin{equation}
\int_\Omega  (\nu \nabla\Yvec - (\Yvec\otimes \yvec + \yvec\otimes \Yvec)):\nabla\wvec - \Pi \nabla\cdot\wvec - \Fvec\cdot\wvec =0, \quad \forall \wvec\in H_0^1(\Omega) 
\end{equation} 
such that  $\Vert \Yvec,\Pi\Vert_{\Hvec_0^1(\Omega)\times L^2(\Omega)} \leq C  ( \Vert \yvec\Vert_{\Hvec^1_0(\Omega)}\Vert + \Vert \Fvec\Vert_{L^2(\Omega)})$
for some $C>0$.
\end{lemma}

Using this lemma for $\Fvec=\vvec$ and $\wvec=\vvec$ ($\vvec$ is the corrector associated to the pair $(\yvec,\pi)$), we obtain that $(\Yvec,\Pi)\in \Hvec_0^1(\Omega)\times L^2(\Omega)$ satisfies $\nabla\cdot\Yvec=0$ and
\begin{equation}\label{eqv}
\int_\Omega  (\nu \nabla\Yvec - (\Yvec\otimes \yvec + \yvec\otimes \Yvec)):\nabla\vvec - \Pi \nabla\cdot\vvec - \vvec\cdot\vvec =0, \quad \forall \wvec\in H_0^1(\Omega) 
\end{equation}

With this pair $(\Yvec,\Pi)$ bounded with respect to $\vvec$ and to $\yvec$, and so with respect to $(\yvec,\pi)$, we have from (\ref{eprime}), (remind that $\nabla \cdot\Yvec=0$)

\begin{equation}
E^{\prime}(\yvec,\pi)\cdot (\Yvec,\Pi)=\int_\Omega -\nu \nabla \vvec \cdot\nabla\Yvec + (\yvec\otimes \Yvec+\Yvec\otimes \yvec):\nabla \vvec+ (\nabla\cdot\vvec)\Pi  d\xvec
\end{equation}
The property $E^{\prime}(\yvec,\pi)\cdot (\Yvec,\Pi)\to 0$ then implies that $\Vert \vvec\Vert_{L^2(\Omega)}\to 0$.  Then, from (\ref{Eprime_Pis}), the property 
$E^{\prime}(\yvec,\pi)\cdot (\vvec,\Pi_s)\to 0$ then implies from the equality (\ref{eqv}) that $\Vert \nabla\vvec\Vert_{L^2(\Omega)}\to 0$.    

Then, \ref{eprime} implies that $\int_{\Omega} \nabla\cdot \yvec  \nabla\cdot\Yvec d\xvec \to 0$   for all $\Yvec\in H_0^1(\Omega)$ so that $\Vert \nabla\cdot\yvec\Vert_{L^2(\Omega)}\to 0$.
\end{proof}

We then recall the following result. 

\begin{lemma}\cite[Corollary 1.8]{pedregalSEMA}
Let $H$ be a Hilbert space. Suppose $J:H\to \mathbb{R}$ is an error functional for which there is a unique $u^0\in H$ with $J(u^0)=0$. If $J(u^j)\to 0$, then $u^j\to u^0$ strongly in $H$, and $J$ complies with the Palais-Smale condition.
\end{lemma}

In view of Proposition \ref{functionalerror} and of the previous Lemma, we have the following  result.
\begin{proposition}\label{strongconvergence}
If $\nu^{-1}\Vert \fvec\Vert_{H^{-1}(\Omega)}$ is small enough, then any minimizing sequence 
$(\yvec_k,\pi_k)_{k>0}$ for $E$ converges strongly toward $(\yvec,\pi)$, unique solution of the boundary value problem (\ref{navierstokes}).
\end{proposition}

\begin{proof} It suffices to apply the previous lemma for $J=E$ and $H=\mathcal{A}$. The uniqueness here follows from Theorem \ref{th-Temam}.
\end{proof}

\begin{remark}
We write again that 
\begin{equation}
E^{\prime}(\yvec,\pi)\cdot (\Yvec,\Pi)=\int_\Omega \nabla \vvec\cdot \nabla \Vvec + (\nabla\cdot \yvec)(\nabla\cdot \Yvec) dx
\end{equation}
where $\Vvec\in \Hvec_0^1(\Omega)$ solves the well-posed formulation

\begin{equation}\label{corrector_unsteadyter}
 \left\{
 \begin{aligned}
 & -\Delta \Vvec + (-\nu \Delta \Yvec + div(\yvec \otimes \Yvec)+ div(\Yvec \otimes \yvec)+ \nabla \Pi) =0, \quad \textrm{in}\quad \Omega,\\
 & \Vvec=0\quad \textrm{on}\quad\partial\Omega.
 \end{aligned}
 \right.
 \end{equation}
 Multiplying the state equation by $\vvec$ and using (\ref{lemma1}), we get
 \begin{equation}
 \begin{aligned}
E^{\prime}(\yvec,\pi)\cdot (\Yvec,\Pi)=&\int_\Omega \nu \Delta \vvec \cdot\Yvec + (\nabla \vvec+(\nabla \vvec)^T) \yvec )\cdot \Yvec + (\nabla\cdot\vvec)\Pi  d\xvec\\
+ & \int_{\Omega} (\nabla\cdot \yvec)(\nabla\cdot \Yvec) d\xvec
\end{aligned}
\end{equation}
so that, at the optimality, the corrector funcion $\vvec$ solves the following linear boundary problem : 
\begin{equation}\label{navierstokes_v}
\left\{
\begin{aligned}
 &      \nu\Delta \vvec +   (\nabla \vvec+(\nabla \vvec)^T) \yvec  - \nabla(\nabla\cdot \yvec)=0, \quad  \nabla\cdot\vvec=0  \quad \text{in}\ \ \Omega \\
  &   \vvec = 0 \quad \text{on}\ \ \partial\Omega
\end{aligned}
\right.
\end{equation}
Proposition \ref{functionalerror} implies that if $E^{\prime}(\yvec,\pi)=0$, then $E(\yvec,\pi)=0$.  Therefore, the formulation (\ref{navierstokes_v}) implies that $\vvec=0$ and $\nabla\cdot \yvec=0$. This may be seen as a unique continuation property for Navier-Stokes equation. 
\end{remark}

\begin{remark}
The previous results hold true if we replace the cost $E$ by 
\begin{equation}
E_{\eps}(\yvec,\pi):=\frac{1}{2}\int_{\Omega} (\vert\nabla \vvec\vert^2 + \vert\nabla\cdot\yvec+\eps \pi\vert^2)\,d\xvec
\end{equation}
for any $\eps>0$, where $\vvec$ solves (\ref{corrector_unsteadybis}).
\end{remark}

As a very interesting consequence, Proposition \ref{strongconvergence} reduces the approximation of the vanishing elements of $E$ to the construction of one arbitrary minimizing sequence, say $(\yvec_k,\pi_k)_{k>0}$, for $E$.

The unsteady case as well as the controllability case are more involved: we refer to \cite{munch-pedregal-NS} based on \cite{EFC-AM-long,Imanuvilov-NS,pedregalNS}.

\end{document}